\documentclass[12pt]{article}%
\usepackage{amsmath}
\usepackage{graphicx}
\usepackage{amsfonts}
\usepackage{amssymb}
\usepackage{latexsym}
\usepackage{color}
\usepackage[colorlinks,pdfpagemode=FullScreen,pdftex,debug]{hyperref}
\usepackage[T1]{fontenc}
\usepackage[sc]{mathpazo}
\usepackage{microtype}%
\newtheorem{theorem}{Theorem}

\newtheorem{conjecture}[theorem]{Conjecture}

\newtheorem{lemma}[theorem]{Lemma}

\newenvironment{proof}[1][Proof]{\noindent{\textbf {#1}  }}  {\hfill$\Box$\bigskip}
\makeatletter
\definecolor {ttlcol}{rgb}{.96,0,.12}
\definecolor {ttlcol}{rgb}{.85,0,0}
\definecolor {pagecol}{rgb}{.98,.98,1}
\definecolor {thmcol}{rgb}{.6,0,0}
\definecolor {deficol}{rgb}{0,.5,.25}
\definecolor {quecol}{rgb}{0,.25,.5}
\definecolor {bgnd}{rgb}{1,1,.95}
\ifx\pdfoutput\relax\let\pdfoutput=\undefined\fi
\newcount\msipdfoutput
\ifx\pdfoutput\undefined\else
\ifcase\pdfoutput\else
\ifx\paperwidth\undefined\else
\ifdim\paperheight=0pt\relax\else\pdfpageheight\paperheight\fi
\ifdim\paperwidth=0pt\relax\else\pdfpagewidth\paperwidth\fi
\fi\fi\fi
\begin{document}

\title{\textbf{On a theorem of Nosal}}
\author{V. Nikiforov\thanks{Department of Mathematical Sciences, University of
Memphis, Memphis TN 38152, USA. Email: \textit{vnikifrv@memphis.edu}} \medskip}
\date{}
\maketitle

\begin{abstract}
Let $G$ be a graph with $m$ edges and spectral radius $\lambda_{1}$. Let
$bk\left(  G\right)  $ stand for the maximal number of triangles with a common
edge in $G$.

In 1970 Nosal proved that if $\lambda_{1}^{2}>m,$ then $G$ contains a
triangle. In this paper we show that the same premise implies that
\[
bk\left(  G\right)  >\frac{1}{12}\sqrt[4]{m}.
\]
This result settles a conjecture of Zhai, Lin, and Shu.

Write $\lambda_{2}$ for the second largest eigenvalue of $G$. Recently, Lin,
Ning, and Wu showed that if $G$ is a triangle-free graph of order at least
three, then
\[
\lambda_{1}^{2}+\lambda_{2}^{2}\leq m,
\]
thereby settling the simplest case of a conjecture of Bollob\'{a}s and the
author. We give a simpler proof of their result.\medskip

\textbf{Keywords:}\textit{ triangle-free graph; spectral radius; graph
booksize; second largest eigenvalue.\smallskip\smallskip}


\end{abstract}

\section{Introduction}

In 1970 Nosal showed that if $G$ is a graph with $m$ edges, and its largest
adjacency eigenvalue $\lambda_{1}$ satisfies
\[
\lambda_{1}^{2}>m,
\]
then $G$ contains a triangle.

During the years this striking and elegant result has attracted significant
attention (see, e.g., \cite{ZLS21} and its references for some highlights.) In
this note we discuss two recent developments of Nosal's result.

The first one concerns the class of subgraphs that are present in $G$ if
$\lambda_{1}^{2}>m$ and $m$ is sufficiently large. \ 

As shown by Zhai, Lin, and Shu in the nice recent paper \cite{ZLS21}, this
class contains graphs other than triangles. These authors studied similar
problems in depth and surveyed some earlier research. In particular, they
raised the following conjecture:

\begin{conjecture}
\label{con1}For every natural number $k$, there exists $m\left(  k\right)  $
such that if $m>m\left(  k\right)  $ and $\lambda_{1}^{2}\geq m,$ then
$bk\left(  G\right)  >k,$ unless $G$ is complete bipartite graph with possibly
some isolated vertices.
\end{conjecture}

In the above conjecture, $bk\left(  G\right)  $ stands for the \emph{booksize}
of $G,$ that is, the maximum number of triangles with a common edge in $G$.
Lower bounds on the booksize are known for Mantel's theorem, but not for the
context of Nosal's inequality.

We confirm Conjecture \ref{con1} by the following theorem:

\begin{theorem}
\label{mt} If $G$ is a graph with $m$ edges and $\lambda_{1}^{2}\geq m,$ then
\[
bk\left(  G\right)  >\frac{1}{12}\sqrt[4]{m},
\]
unless $G$ is a complete bipartite graph with possibly some isolated vertices.
\end{theorem}

The other main result of our note concerns the following conjecture of
Bollob\'{a}s and the author \cite{BoNi07}:

\begin{conjecture}
\label{con2} Let $G$ be a graph with $m$ edges, at least $r+1$ vertices, and
second largest eigenvalue $\lambda_{2}$. If $G$ is $K_{r+1}$-free, then
\begin{equation}
\lambda_{1}^{2}+\lambda_{2}^{2}\leq2\left(  1-\frac{1}{r}\right)  m.
\label{BN}%
\end{equation}

\end{conjecture}

Recently, Lin, Ning, and Wu \cite{LNW21} settled the case $r=2$ of Conjecture
\ref{con2} by a clever argument using majorization theory.

Before stating their result, recall that a \emph{blow-up} of a graph $H$ is
obtained by replacing each vertex $v$ of $H$ by an independent set $B_{v}$ and
replacing each edge $\left\{  u,v\right\}  $ of $H$ by a complete bipartite
graph with vertex classes $B_{u}$ and $B_{v}$. Also, $P_{k}$ stands for the
path of order $k.$

\begin{theorem}
[Lin, Ning, Wu]\label{th2}Let $G$ be a graph with $m$ edges, of order at least
$3,$ and let $\lambda_{2}$ be its second largest adjacency eigenvalue. If $G$
is triangle-free, then
\[
\lambda_{1}^{2}+\lambda_{2}^{2}\leq m.
\]
Equality holds if and only if $G$ is a blow-up of $P_{2}\cup K_{1},$
$2P_{2}\cup K_{1},$ $P_{4}\cup K_{1},$ or $P_{5}\cup K_{1}.$
\end{theorem}

In Section \ref{BNS} we give a simple straightforward proof of Theorem
\ref{th2}.

Theorem \ref{mt} is proved in Sections \ref{sups} and \ref{mtS}.

\section{Notation and preliminary results}

A\emph{ }$k$\emph{-walk} stands for a walk on $k$ vertices, that is, a walk of
length $k-1$.

Given a graph $G$, we write:\medskip

- $V\left(  G\right)  $ for the set of its vertices and $E\left(  G\right)  $
for the set of its edges;\medskip

- $e\left(  G\right)  $ for the number of its edges;\medskip

- $t\left(  G\right)  $ for the number of its triangles;\medskip

- $t^{\prime\prime}\left(  G\right)  $ for the number of its induced subgraphs
of order three with exactly one edge;\medskip

- $w_{k}\left(  G\right)  $ for the number of its $k$-walks;\medskip

- $\rho\left(  G\right)  $ for its largest adjacency eigenvalue.\medskip

For a vertex $u\in V\left(  G\right)  $, we write:\medskip

- $N\left(  u\right)  $ for the set of its neighbors;\medskip

- $\overline{N}\left(  u\right)  $ for the set $V\left(  G\right)  \backslash
N\left(  u\right)  $;\medskip

- $w_{k}\left(  u\right)  $ for the number of $k$-walks starting with
$u$;\medskip

- $t\left(  u\right)  $ for the number of triangles containing $u$;\medskip

- $t^{\prime\prime}\left(  u\right)  $ for the number of edges $\left\{
v,w\right\}  $ of $G$ such that $u\in\overline{N}\left(  u\right)
\cap\overline{N}\left(  v\right)  .\medskip$

Note, in particular, that
\[%
{\textstyle\sum\limits_{u\in V\left(  G\right)  }}
t^{\prime\prime}\left(  u\right)  =t^{\prime\prime}\left(  G\right)  .
\]

For a proof of the following theorem of Wei \cite{Wei52} we refer the reader
to \cite{CRS10}, p. 182.

\begin{theorem}
[\textbf{Wei}]\label{tWei}Let $G$ be a connected nonbipartite graph of order
$n$ and let $\left(  x_{1},\ldots,x_{n}\right)  $ be a positive eigenvector to
$\rho\left(  G\right)  .$ For every vertex $u\in V\left(  G\right)  ,$%
\[
\lim_{k\rightarrow\infty}\frac{w_{k}\left(  u\right)  }{w_{k}\left(  G\right)
}=\frac{x_{u}}{x_{1}+\cdots+x_{n}}.
\]

\end{theorem}

The following inequality is the instance $r=2$ of Theorem 1 in \cite{BoNi07}:

\begin{lemma}
\label{BoNiL}If $G$ is a graph with $e\left(  G\right)  =m,$ $\rho\left(
G\right)  =\rho,$\ and $t\left(  G\right)  =t,$ then%
\[
3t\geq\rho^{3}-\rho m.
\]

\end{lemma}

\section{\label{sups}Results supporting the proof of Theorem \ref{mt}}

Our proof of Theorem \ref{mt} is quite involved, so we first describe its key points.

We start by taking a nonnegative unit eigenvector $\left(  x_{1},\ldots
,x_{n}\right)  $ to $\rho\left(  G\right)  $. If there is an edge $\left\{
i,j\right\}  \in E\left(  G\right)  $ with
\[
x_{i}x_{j}<\frac{1}{8\sqrt{e\left(  G\right)  }},
\]
we remove it and show that the graph $G^{\prime}=G-\left\{  i,j\right\}  $
satisfies
\[
\rho\left(  G^{\prime}\right)  >\sqrt{e\left(  G^{\prime}\right)  }+\frac
{1}{4\sqrt{e\left(  G\right)  }}.
\]
If such edge removals are carried out sufficiently long, the difference
$\rho\left(  G^{\prime}\right)  -\sqrt{e\left(  G^{\prime}\right)  }$ becomes
large enough to entail that $bk\left(  G^{\prime}\right)  $ is as large as needed..

If on the contrary, at some stage it appears that
\[
x_{i}x_{j}\geq\frac{1}{8\sqrt{e\left(  G^{\prime}\right)  }}%
\]
for all edges $\left\{  i,j\right\}  \in E\left(  G^{\prime}\right)  ,$ it
turns out that $G^{\prime}$ may have only one nontrivial component and is
nonbipartite. Under these premises for $G^{\prime}$, Theorem \ref{supt}
implies that $bk\left(  G^{\prime}\right)  $ is as large as needed.

Theorem \ref{supt} itself is based on two rather technical results--Theorem
\ref{th1} and Lemma \ref{colem}.

\begin{theorem}
\label{th1}Let $G$ be a connected nonbipartite graph of order $n$ with
\[
e\left(  G\right)  =m,\text{ \ }\rho\left(  G\right)  =\rho,\text{
\ \ }t\left(  G\right)  =t,\text{ \ \ }t^{\prime\prime}\left(  G\right)
=t^{\prime\prime}.
\]
Suppose that $\left(  x_{1},\ldots,x_{n}\right)  $ is a positive unit
eigenvector to $\rho$ and let%
\[
c=\frac{1}{x_{1}+\cdots+x_{n}}\min\left\{  x_{1},\ldots,x_{n}\right\}  .
\]
Then
\[
\rho^{3}-\rho m+c\rho t^{\prime\prime}\leq3t.
\]

\end{theorem}

\begin{proof}
For every edge $\left\{  u,v\right\}  \in E\left(  G\right)  ,$ obviously%
\begin{align*}
w_{k}\left(  G\right)   &  =%
{\textstyle\sum\limits_{w\in N\left(  u\right)  \cup N\left(  v\right)  }}
w_{k}\left(  w\right)  +%
{\textstyle\sum\limits_{w\in\overline{N}\left(  u\right)  \cap\overline
{N}\left(  v\right)  }}
w_{k}\left(  w\right) \\
&  =%
{\textstyle\sum\limits_{w\in N\left(  u\right)  }}
w_{k}\left(  w\right)  +%
{\textstyle\sum\limits_{w\in N\left(  v\right)  }}
w_{k}\left(  w\right)  -%
{\textstyle\sum\limits_{w\in N\left(  u\right)  \cap N\left(  v\right)  }}
w_{k}\left(  w\right)  +%
{\textstyle\sum\limits_{w\in\overline{N}\left(  u\right)  \cap\overline
{N}\left(  v\right)  }}
w_{k}\left(  w\right) \\
&  =w_{k+1}\left(  u\right)  +w_{k+1}\left(  v\right)  -%
{\textstyle\sum\limits_{w\in N\left(  u\right)  \cap N\left(  v\right)  }}
w_{k}\left(  w\right)  +%
{\textstyle\sum\limits_{w\in\overline{N}\left(  u\right)  \cap\overline
{N}\left(  v\right)  }}
w_{k}\left(  w\right)  .
\end{align*}
Summing this identity over all edges $\left\{  u,v\right\}  \in E\left(
G\right)  ,$ we get%
\begin{align*}
w_{k}\left(  G\right)  m  &  =%
{\textstyle\sum\limits_{u\in V\left(  G\right)  }}
w_{k+1}\left(  u\right)  d\left(  u\right)  -%
{\textstyle\sum\limits_{u\in V\left(  G\right)  }}
t\left(  u\right)  w_{k}\left(  u\right)  +%
{\textstyle\sum\limits_{u\in V\left(  G\right)  }}
t^{\prime\prime}\left(  u\right)  w_{k}\left(  u\right) \\
&  =w_{k+2}\left(  G\right)  -%
{\textstyle\sum\limits_{u\in V\left(  G\right)  }}
t\left(  u\right)  w_{k}\left(  u\right)  +t^{\prime\prime}\left(  u\right)
w_{k}\left(  u\right)  .
\end{align*}
Since for any vertex $u$ we have $w_{k}\left(  u\right)  \leq w_{k-1}\left(
G\right)  ,$ it follows that
\[%
{\textstyle\sum\limits_{u\in V\left(  G\right)  }}
t\left(  u\right)  w_{k}\left(  u\right)  \leq%
{\textstyle\sum\limits_{u\in V\left(  G\right)  }}
t\left(  u\right)  w_{k-1}\left(  G\right)  =3tw_{k-1}\left(  G\right)  .
\]
Hence,%
\[
w_{k}\left(  G\right)  m\geq w_{k+2}\left(  G\right)  -3tw_{k-1}\left(
G\right)  +%
{\textstyle\sum\limits_{u\in V\left(  G\right)  }}
t^{\prime\prime}\left(  u\right)  w_{k}\left(  u\right)  .
\]
Dividing both sides by $w_{k-1}\left(  G\right)  ,$ we get
\[
\frac{w_{k}\left(  G\right)  }{w_{k-1}\left(  G\right)  }m\geq\frac
{w_{k+2}\left(  G\right)  }{w_{k-1}\left(  G\right)  }-3t+\frac{w_{k}\left(
G\right)  }{w_{k-1}\left(  G\right)  }%
{\textstyle\sum\limits_{u\in V\left(  G\right)  }}
t^{\prime\prime}\left(  u\right)  \frac{w_{k}\left(  u\right)  }{w_{k}\left(
G\right)  }%
\]
The formula for the number of $k$-walks by Cvekovi\'{c} \cite{Cve70} (see also
\cite{CRS10}, p. 15) implies that
\[
\lim_{k\rightarrow\infty}\frac{w_{k}\left(  G\right)  }{w_{k-1}\left(
G\right)  }=\rho\text{ \ and \ }\lim_{k\rightarrow\infty}\frac{w_{k+2}\left(
G\right)  }{w_{k-1}\left(  G\right)  }=\rho^{3}.
\]
Hence, in view of Theorem \ref{tWei}, we see that
\begin{align*}
m\rho &  \geq\rho^{3}-3t+\rho%
{\textstyle\sum\limits_{u\in V\left(  G\right)  }}
t^{\prime\prime}\left(  u\right)  \frac{x_{u}}{x_{1}+\cdots+x_{n}}\\
&  \geq\rho^{3}-3t+c\rho t^{\prime\prime}.
\end{align*}
Theorem \ref{th1} is proved.
\end{proof}

\begin{lemma}
\label{colem}If $G$ is a graph of order $n$ with $bk\left(  G\right)  =\beta$,
then
\[
\left(  n-3\beta\right)  t\left(  G\right)  \leq\beta t^{\prime\prime}\left(
G\right)  .
\]

\end{lemma}

\begin{proof}
For $0\leq j\leq2,$ write $k_{4}^{(j)}$ for the number of induced subgraphs of
$G$ that are isomorphic to a triangle together with an additional vertex
joined to precisely $j$ vertices of the triangle. E.g., $k_{4}^{\left(
0\right)  }\left(  G\right)  $ is the number of induced subgraphs of $G$ that
are isomorphic to a triangle with an isolated vertex. In addition, we write
$k_{4}$ for the number of $4$-cliques of $G.$

It is not hard to check the following three relations:%
\begin{equation}
\left(  n-3\right)  t=k_{4}^{\left(  0\right)  }+k_{4}^{\left(  1\right)
}+2k_{4}^{\left(  2\right)  }+4k_{4}. \label{in1}%
\end{equation}%
\begin{equation}
3\left(  \beta-1\right)  t\geq%
{\textstyle\sum\limits_{\left\{  u,v\right\}  \in E\left(  G\right)  }}
\left\vert N\left(  u\right)  \cap N\left(  v\right)  \right\vert \left(
\left\vert N\left(  u\right)  \cap N\left(  v\right)  \right\vert -1\right)
=2k_{4}^{\left(  2\right)  }+12k_{4}. \label{in2}%
\end{equation}%
\begin{equation}
\beta t^{\prime\prime}\geq%
{\textstyle\sum\limits_{\left\{  u,v\right\}  \in E\left(  G\right)  }}
\left\vert N\left(  u\right)  \cap N\left(  v\right)  \right\vert
|\overline{N}\left(  u\right)  \cap\overline{N}\left(  v\right)  |\text{
}=k_{4}^{\left(  1\right)  }+3k_{4}^{\left(  0\right)  }. \label{in3}%
\end{equation}

Now, subtracting (\ref{in2}) from (\ref{in1}), in view of (\ref{in3}), we find
that
\begin{align*}
nt-3\beta t  &  \leq k_{4}^{\left(  0\right)  }+k_{4}^{\left(  1\right)
}+2k_{4}^{\left(  2\right)  }+4k_{4}-2k_{4}^{\left(  2\right)  }-12k_{4}\\
&  =k_{4}^{\left(  0\right)  }+k_{4}^{\left(  1\right)  }-8k_{4}\\
&  \leq\beta t^{\prime\prime}.
\end{align*}

\end{proof}

Having Theorem \ref{th1} and Lemma \ref{colem} in hand, we are ready to prove
a statement similar to Theorem \ref{mt} under some extra assumptions.

\begin{theorem}
\label{supt}Let $G$ be a connected graph with $m$ edges such that $\rho\left(
G\right)  >\sqrt{m}.$ Suppose that $\left(  x_{1},\ldots,x_{n}\right)  $ is a
positive unit eigenvector to $\rho\left(  G\right)  .$ If%
\[
x_{i}x_{j}\geq\frac{1}{8\sqrt{m}}%
\]
for every edge $\left\{  i,j\right\}  \in E\left(  G\right)  $, then
\[
bk\left(  G\right)  >\frac{1}{12}\sqrt[4]{2m}.
\]

\end{theorem}

\begin{proof}
Let $\rho=\rho\left(  G\right)  ,$ let $i$ be a vertex, and suppose that
$\left\{  i,j\right\}  \in E\left(  G\right)  $. We have
\[
\rho x_{i}^{2}\geq x_{j}x_{j}>\frac{1}{8\sqrt{m}}>\frac{1}{8\rho}.
\]
Now, letting
\[
c=\frac{1}{x_{1}+\cdots+x_{n}}\min\left\{  x_{1},\ldots,x_{n}\right\}  ,
\]
we see that
\[
c\rho>\frac{1}{\sqrt{8}}\cdot\frac{1}{x_{1}+\cdots+x_{n}}\geq\frac{1}%
{\sqrt{8n}}.
\]
Note that $G$ is nonbipartite as $\rho>\sqrt{m}$. Hence, Theorem \ref{th1}
implies that
\[
t^{\prime\prime}<\sqrt{8n}\left(  3t+\rho m-\rho^{3}\right)  .
\]
Combining this inequality with Lemma \ref{colem} and letting $\beta=bk\left(
G\right)  $, we find that%
\begin{align*}
\left(  n-3\beta\right)  t  &  \leq\beta t^{\prime\prime}<\beta\sqrt
{8n}\left(  3t+\rho m-\rho^{3}\right) \\
\left(  n-3\beta-3\beta\sqrt{8n}\right)  t  &  <\beta\sqrt{8n}\left(
m-\rho^{2}\right)  \rho\leq0.
\end{align*}
Therefore, $n-3\beta-3\beta\sqrt{8n}<0,$ and we see that
\[
\beta>\frac{n}{3\left(  1+\sqrt{8n}\right)  }>\frac{n}{3\left(  4\sqrt
{n}\right)  }=\frac{\sqrt{n}}{12}\geq\frac{1}{12}\sqrt[4]{2m}.
\]
The proof of Theorem \ref{supt} is completed.
\end{proof}

\section{\label{mtS}Proof of Theorem \ref{mt}}

We may suppose that $G$ has no isolated vertices.

We first describe a simple procedure that constructs a sequence of graphs
\[
G_{0}\supset G_{1}\supset\cdots\supset G_{l}%
\]
such that $V\left(  G_{i}\right)  =V\left(  G\right)  $ and $e\left(
G_{i}\right)  =m-i$ for every $i=1,\ldots,l.\medskip$

\textbf{Step 1} Set\textbf{ }$l:=0$ and $G_{0}:=G$.$\medskip$

\textbf{Step 2 }If $l=\left\lceil m/2\right\rceil $, stop.\medskip

\textbf{Step 3 }Let $\mathbf{x}_{l}=\left(  x_{1},\ldots,x_{n}\right)  $ be a
nonnegative unit eigenvector to $\rho\left(  G_{l}\right)  $.\medskip

\textbf{Step 4 }If there is an edge $\left\{  i,j\right\}  \in E\left(
G_{l}\right)  $ with%
\[
x_{i}x_{j}<\frac{1}{8\sqrt{m-l}},
\]
set
\begin{align*}
G_{l+1}  &  :=G_{l}-\left\{  i,j\right\} \\
l  &  :=l+1
\end{align*}
and iterate the procedure from step 2.\medskip

\textbf{Step 5 }If there is no such edge, stop.\medskip

Let $k$ be the number of the last graph constructed by the procedure. Note
that for every $s=1,\ldots,k$%
\begin{equation}
\rho\left(  G_{s}\right)  >\rho\left(  G_{s-1}\right)  -\frac{1}{4\sqrt
{m-s+1}}. \label{rin}%
\end{equation}
Indeed, let $\mathbf{x}_{s-1}=\left(  x_{1},\ldots,x_{n}\right)  $ be the unit
eigenvector to $\rho_{s-1}$ and let $\left\{  i,j\right\}  \in E\left(
G_{s-1}\right)  $ be the edge such that
\[
G_{s}=G_{s-1}-\left\{  i,j\right\}  ,
\]
which entails%
\[
x_{i}x_{j}<\frac{1}{8\sqrt{m-s+1}}.
\]
The Rayleigh principle implies that
\begin{align*}
\rho\left(  G_{s}\right)   &  \geq2%
{\textstyle\sum\limits_{\left\{  u,v\right\}  \in E\left(  G_{s}\right)  }}
x_{u}x_{v}=-2x_{i}x_{j}+2%
{\textstyle\sum\limits_{\left\{  u,v\right\}  \in E\left(  G_{s}\right)  }}
x_{u}x_{v}\\
&  =\rho\left(  G_{s-1}\right)  -2x_{i}x_{j}>\rho_{s-1}-\frac{1}{4\sqrt
{m-s+1}}.
\end{align*}
Now, adding inequalities (\ref{rin}) for $s=1,\ldots,k$, we get
\begin{align*}
\rho\left(  G_{k}\right)   &  \geq\rho\left(  G_{0}\right)  -\frac{1}%
{4\sqrt{m}}-\cdots-\frac{1}{4\sqrt{m-k+1}}\\
&  \geq\sqrt{m-k}+\sqrt{m}-\sqrt{m-k}-\frac{k}{4\sqrt{m-k+1}}\\
&  >\sqrt{m-k}+\frac{k}{2\sqrt{m}}-\frac{k}{4\sqrt{m/2}}\\
&  =\sqrt{m-k}+\frac{k}{4\sqrt{m}}\left(  2-\sqrt{2}\right)  .
\end{align*}

Suppose that the procedure stops because $k=\left\lceil m/2\right\rceil .$
Then we have%
\[
\rho\left(  G_{k}\right)  >\sqrt{m-k}+\frac{1}{4}\left(  \sqrt{2}-1\right)
\sqrt{m-k},
\]
and in view of
\[
bk\left(  G_{k}\right)  \left(  m-k\right)  \geq3t\left(  G_{k}\right)  ,
\]
Lemma \ref{BoNiL} implies that%
\begin{align*}
bk\left(  G_{k}\right)  \left(  m-k\right)   &  \geq\rho\left(  G_{k}\right)
\left(  \rho^{2}\left(  G_{k}\right)  -\left(  m-k\right)  \right) \\
&  \geq\left(  m-k\right)  \sqrt{m-k}\left(  \left(  1+\frac{1}{4}\left(
\sqrt{2}-1\right)  \right)  ^{2}-1\right) \\
&  >\frac{1}{5}\left(  m-k\right)  \sqrt{m-k}.
\end{align*}
Now, we see that
\[
bk\left(  G\right)  \geq bk\left(  G_{k}\right)  >\frac{1}{5}\sqrt
{\left\lfloor m/2\right\rfloor }>\frac{1}{12}\sqrt[4]{m},
\]
completing the proof of Theorem \ref{mt} in the case $k=\left\lceil
m/2\right\rceil $.

Next, suppose that the procedure stops because
\begin{equation}
x_{i}x_{j}\geq\frac{1}{8\sqrt{m-k}} \label{elb}%
\end{equation}
for every edge $\left\{  i,j\right\}  \in E\left(  G_{k}\right)  $.

Let us drop all isolated vertices that $G_{k}$ may have and write
$G_{k}^{\prime}$ for the resulting graph. Let $\left(  x_{1},\ldots
,x_{p}\right)  $ be the restriction of $\mathbf{x}_{k}=\left(  x_{1}%
,\ldots,x_{n}\right)  $ to the vertices of $G_{k}^{\prime}.$

Inequality (\ref{elb}) implies that $\left(  x_{1},\ldots,x_{p}\right)  $ is
positive. We shall show that $G_{k}^{\prime}$ is connected. Indeed, since
$\left(  x_{1},\ldots,x_{p}\right)  $ is positive, the spectral radius of each
component of $G_{k}^{\prime}$ is equal to $\rho\left(  G_{k}^{\prime}\right)
.$ If $G$ has more than one component, let $C$ be a component of
$G_{k}^{\prime}$ with smallest number of edges. We see that
\[
2e\left(  C\right)  \leq e\left(  G_{k}^{\prime}\right)  \leq\rho\left(
G_{k}^{\prime}\right)  =\rho\left(  C\right)  ,
\]
which is a contradiction. Hence, $G_{k}^{\prime}$ is connected$.$

If $k\geq1$, we have
\[
\rho\left(  G_{k}^{\prime}\right)  >\sqrt{m-k},
\]
so $G_{k}^{\prime}$ is nonbipartite. Now, Theorem \ref{th2} implies that%
\[
bk\left(  G\right)  \geq bk\left(  G_{k}^{\prime}\right)  >\frac{1}%
{12}\sqrt[4]{2\left(  m-k\right)  }.
\]
In view of $k\leq\left\lfloor m/2\right\rfloor ,$ we get
\[
bk\left(  G\right)  \geq\frac{1}{12}\sqrt[4]{m}.
\]

It remains the case $k=0$, that is, $G_{k}=G$. We assumed that $G$ has no
isolated vertices, and we showed above that $\rho\left(  G\right)  \geq m$
implies that $G$ is connected. Hence, if $G$ is nonbipartite, then Theorem
\ref{supt} implies that
\[
bk\left(  G\right)  >\frac{1}{12}\sqrt[4]{2m}>\frac{1}{12}\sqrt[4]{m}.
\]

Finally, if $G$ is bipartite, then $\rho\left(  G\right)  =\sqrt{m}$ implies
that $G$ is complete bipartite.

Theorem \ref{mt} is proved.

\section{\label{BNS}Proof of Theorem \ref{th2}}

In this section we prove Theorem \ref{th2}. Our proof is based on a simple
analytic result:

\begin{lemma}
\label{vecl}Let $k\geq3,$ and $a,b,x_{1},\ldots,x_{k}$ be nonnegative numbers
such that
\[
b\leq a,\text{ }x_{1}\leq a,\text{ }\ldots\text{ },\text{ }x_{k}\leq a.
\]
If%
\[
x_{1}^{2}+\text{ }\cdots\text{ }+x_{k}^{2}\leq a^{2}+b^{2},
\]
then for every real $p>2,$%
\[
x_{1}^{p}+\text{ }\cdots\text{ }+x_{k}^{p}<a^{p}+b^{p},
\]
unless%
\[
x_{1}=a,\text{ }x_{2}=b,\text{ }x_{3}=\text{ }\cdots\text{ }=x_{k}=0.
\]

\end{lemma}

\begin{proof}
We may suppose that $a>0,$ as otherwise the assertion is trivially
true.\textbf{ }

Fix $a>0$ and $b\geq0$, and write $X_{a,b}$ for the compact set of all vectors
$\left(  y_{1},\ldots,y_{k}\right)  $ satisfying%
\begin{align*}
0  &  \leq y_{i}\leq a,\text{ \ \ \ }\left(  1\leq i\leq k\right) \\
y_{1}^{2}+\cdots+y_{k}^{2}  &  \leq a^{2}+b^{2}.
\end{align*}

Let the continuous function $y_{1}^{p}+\cdots+y_{k}^{p}$ attains maximum over
$X_{a,b}$ at $\left(  x_{1},\ldots,x_{k}\right)  $, and suppose by symmetry
that
\[
x_{1}\leq x_{2}\leq\cdots\leq x_{k}.
\]
We shall show that $x_{k}=a$. Assume for contradiction that $x_{k}<a.$ Then we
have%
\[
x_{1}^{2}+\cdots+x_{k}^{2}=a^{2}+b^{2},
\]
as otherwise we can increase $x_{k}$ by a tiny bit, keeping the resulting
vector in $X$ and increasing $x_{1}^{p}+$ $\cdots$ $+x_{k}^{p},$ which
contradicts the choice of $\left(  x_{1},\ldots,x_{k}\right)  $.

Since $x_{k}<a,$ we find that $x_{k-1}>0,$ for otherwise
\[
x_{k-1}=\cdots=x_{1}=0,
\]
and so
\[
x_{1}^{2}+\cdots+x_{k}^{2}<a^{2}+b^{2}.
\]

Since the function $x^{2}$ is a homeomorphism for $x>0$, we can find
$\varepsilon>0$ and $\delta>0$ such that
\[
x_{k-1}-\varepsilon>0,\text{ \ \ }x_{k}+\delta<a,
\]
and
\[
\left(  x_{k-1}-\varepsilon\right)  ^{2}+\left(  x_{k}+\delta\right)
^{2}=x_{k-1}^{2}+x_{k}^{2}.
\]
Therefore, the $k$-vector
\[
(x_{1},\text{ }\ldots\text{ },\text{ }x_{k-2,}\text{ }x_{k-1}-\varepsilon
,\text{ }x_{k}+\delta)
\]
belongs to $X_{a,b}.$

Note that for $z>0$ the function $z^{p/2}$ is strictly convex, as its second
derivative
\[
\frac{p\left(  p-2\right)  }{4}z^{p/2-2}%
\]
is positive. Hence, if $0<\alpha<z_{1}\leq z_{2},$ we have%
\[
\left(  z_{1}-\alpha\right)  ^{p/2}+\left(  z_{2}+\alpha\right)  ^{p/2}%
>z_{1}^{p/2}+z_{2}{}^{p/2}.
\]

Now setting
\begin{align*}
z_{1} &  =x_{k-1}^{2},\\
z_{2} &  =x_{k}^{2},\\
\alpha &  =x_{k-1}^{2}-\left(  x_{k-1}-\varepsilon\right)  ^{2}=\left(
x_{k}+\delta\right)  ^{2}-x_{k}^{2},
\end{align*}
we see that%
\begin{align*}
\left(  x_{k-1}-\varepsilon\right)  ^{p}+\left(  x_{k}+\delta\right)  ^{p} &
=\left(  z_{1}-\alpha\right)  ^{p/2}+\left(  z_{2}+\alpha\right)  ^{p/2}\\
&  >z_{1}^{p/2}+z_{2}{}^{p/2}\\
&  =x_{k-1}^{p}+x_{k}^{p}.
\end{align*}
This inequality contradicts the assumption that $y_{1}^{p}+\cdots+y_{k}^{p}$
attains maximum over $X_{a,b}$ at $\left(  x_{1},\ldots,x_{k}\right)  ,$ and
therefore $x_{k}=a$.

Further, we see that
\[
x_{1}^{2}+\cdots+x_{k-1}^{2}\leq b^{2},
\]
which yields%
\[
x_{1}\leq b,\text{ }\ldots\text{ },\text{ }x_{k-1}\leq b.
\]
Therefore,%
\begin{align*}
b^{p} &  \geq x_{1}^{2}b^{p-2}+\cdots+x_{k-1}^{2}b^{p-2}\\
&  \geq x_{1}^{p}+\cdots+x_{k-1}^{p}.
\end{align*}
Equality may hold only if
\[
x_{k-1}=b\text{ \ \ and \ }x_{k-2}=\cdots=x_{1}=0.
\]
Hence,%
\begin{align*}
a^{p}+b^{p} &  \geq x_{1}^{p}+\cdots+x_{k-1}^{p}+a^{p}\\
&  =x_{1}^{p}+\cdots+x_{k}^{p},
\end{align*}
completing the proof of the lemma.
\end{proof}

\begin{proof}
[\textbf{Proof of Theorem \ref{th2}}]Let $\lambda_{1}\geq\cdots\geq\lambda
_{n}$ be the adjacency eigenvalues of $G.$

We first prove the contrapositive of the statement of the theorem: if
\[
\lambda_{1}^{2}+\lambda_{2}^{2}>m,
\]
then $G$ has a triangle.

Clearly, we may assume that $G$ is a noncomplete graph and therefore
$\lambda_{2}\geq0.$ Since
\[
\lambda_{1}^{2}+\cdots+\lambda_{n}^{2}=2m,
\]
we see that%
\begin{equation}
\lambda_{1}^{2}+\lambda_{2}^{2}>m>\lambda_{3}^{2}+\cdots+\lambda_{n}%
^{2}.\label{inl}%
\end{equation}

Let%
\begin{align*}
k  &  =n-2,\\
a  &  =\lambda_{1},\\
b  &  =\lambda_{2},\\
x_{i}  &  =\left\vert \lambda_{i+2}\right\vert ,\text{ \ \ }i=1,\ldots,k,\\
p  &  =3.
\end{align*}

Since $a\geq b$ and $a\geq x_{i}$ for $i=1,\ldots,k,$ Lemma \ref{vecl} implies
that
\[
\lambda_{1}^{3}+\lambda_{2}^{3}\geq\left\vert \lambda_{3}\right\vert
^{3}+\cdots+\left\vert \lambda_{n}\right\vert ^{3}.
\]

Note that equality cannot hold above, for otherwise Lemma \ref{vecl} implies
that
\[
\lambda_{1}^{2}+\lambda_{2}^{2}=\lambda_{3}^{2}+\cdots+\lambda_{n}^{2},
\]
contradicting (\ref{in2}).

Hence,%
\[
6t\left(  G\right)  =\lambda_{1}^{3}+\lambda_{2}^{3}+\lambda_{3}^{3}%
+\cdots+\lambda_{n}^{3}\geq\lambda_{1}^{3}+\lambda_{2}^{3}-\left\vert
\lambda_{3}\right\vert ^{3}-\cdots-\left\vert \lambda_{n}\right\vert ^{3}>0.
\]
Therefore, $G$ contains a triangle, proving the inequality of Theorem
\ref{th2}.

If $G$ is triangle-free and
\[
\lambda_{1}^{2}+\lambda_{2}^{2}=m,
\]
setting $k,a,b,x_{i},p$ as above, Lemma \ref{vecl} implies that
\[
\lambda_{1}^{3}+\lambda_{2}^{3}=\left\vert \lambda_{3}\right\vert ^{3}%
+\cdots+\left\vert \lambda_{n}\right\vert ^{3},
\]
and therefore, the condition for equality in Lemma \ref{vecl} implies that
\begin{align*}
\lambda_{1}^{2}  &  =\lambda_{n}^{2},\\
\lambda_{2}^{2}  &  =\lambda_{n-1}^{2},\\
\lambda_{3}  &  =\cdots=\lambda_{n-2}=0.
\end{align*}
Now the condition for equality in Theorem \ref{th2} follows from a result of
Oboudi \cite{Obo16}, exactly as in \cite{LNW21}.
\end{proof}

\section{Concluding remarks}

The bound on $bk\left(  G\right)  $ given by Theorem \ref{mt} seems far from
optimal, as the multiplicative constant and perhaps the exponent $1/4$ can be improved.

It seems unlikely that Lemma \ref{vecl} can be extended to support the proof
of Conjecture \ref{con2} for $r\geq3.$

\end{document}